\documentclass[10pt,reqno]{amsart}
\usepackage{amsmath,amscd,amsfonts,amsthm,amssymb,latexsym,mathrsfs}
\usepackage{hyperref}
\hypersetup{
    bookmarks=true,         
    unicode=false,          
    pdftoolbar=true,        
    pdfmenubar=true,        
    pdffitwindow=false,     
    pdfstartview={FitH},    
    pdftitle={My title},    
    pdfauthor={Author},     
    pdfsubject={Subject},   
    pdfcreator={Creator},   
    pdfproducer={Producer}, 
    pdfkeywords={keyword1} {key2} {key3}, 
    pdfnewwindow=true,      
    colorlinks=true,       
    linkcolor=blue,          
    citecolor=blue,        
    filecolor=blue,      
    urlcolor=gray           
}
\usepackage{tikz}
\usetikzlibrary{matrix,arrows,patterns}
\usepackage{url}

\theoremstyle{plain}
   \newtheorem{theorem}{Theorem}[section]
   
   \newtheorem{lemma}[theorem]{Lemma}
   \newtheorem{corollary}[theorem]{Corollary}
   
   \newtheorem{conjecture}[theorem]{Conjecture}
   
\theoremstyle{definition}
   
   \newtheorem{example}{Example}

\theoremstyle{remark}
   \newtheorem{remark}[theorem]{Remark}

\author[P.~Br\"and\'en]{Petter Br\"and\'en}
\address{Department of Mathematics, Royal Institute of Technology, SE-100 44 Stockholm,
Sweden}
\email{pbranden@kth.se}



\thanks{PB is a Royal Swedish Academy of Sciences Research Fellow
  supported by a grant from the Knut and Alice Wallenberg
  Foundation. The research is also supported by the G\"oran Gustafsson Foundation.}
\def\kkk{\kern.2ex\mbox{\raise.5ex\hbox{{\rule{.35em}{.12ex}}}}\kern.2ex}

\numberwithin{equation}{section}

\newcommand{\xx}{\mathbf{x}}
\newcommand{\vv}{\mathbf{v}}
\newcommand{\yy}{\mathbf{y}}
\newcommand{\zz}{\mathbf{z}}

\newcommand{\ee}{\mathbf{e}}

\newcommand{\one}{\mathbf{1}}

\newcommand{\PP}{\mathcal{P}}

\newcommand{\ZZ}{\mathbb{Z}}

\newcommand{\RR}{\mathbb{R}}
\newcommand{\CC}{\mathbb{C}}

\def\newop#1{\expandafter\def\csname #1\endcsname{\mathop{\rm
#1}\nolimits}}

\newop{per}
\newop{diag}
\newop{supp}
\newop{Proj}
\newop{JP}
\newop{rank}
\newop{Sym}
\newop{sign}
\newop{Int}
\newop{Cap}
\newop{PolO}
\newop{St}
\newop{glb}
\newop{MAP}
\newop{MOD}

\begin{document}

\title[Elementary representations]{Hyperbolicity cones of elementary symmetric polynomials are spectrahedral}

\maketitle

\begin{abstract}
We prove that the hyperbolicity cones of elementary symmetric polynomials are spectrahedral, i.e., they are slices of the cone of positive semidefinite matrices. The proof uses the matrix--tree theorem, an idea already present in 
Choe \emph{et al.}
\end{abstract}

\section{Introduction and main results}
A homogenous polynomial $h(\xx) \in \RR[x_1,\ldots, x_n]$ is \emph{hyperbolic} with respect to a vector $\ee \in \RR^n$ if $h(\ee) \neq 0$ and for all $\xx \in \RR^n$, the univariate polynomial $t \mapsto h(\xx +t\ee)$ has only real zeros. The \emph{hyperbolicity cone}, $\Lambda_{+}(h,\ee)$,  is the closure of the connected component of $\{ \xx \in \RR^n : h(\xx) \neq 0\}$ which contains $\ee$. Hyperbolicity cones are convex, and if $\ee'$ is in the interior of $\Lambda_{+}(h,\ee)$,  then $h$ is hyperbolic with respect to $\ee'$ and $\Lambda_{+}(h,\ee)=  \Lambda_{+}(h,\ee')$, see e.g. \cite{Ga,Gul,Ren}. The notion of hyperbolic polynomials originates from PDE--theory and the work of Petrovsky and G\aa rding. However, during the last fifteen years there has been increasing interest in hyperbolic polynomials from unexpected areas such as combinatorics and convex optimization \cite{COSW,Gul,Ren}. Optimization over hyperbolicity cones was first considered by G\"uler \cite{Gul} and a  rich theory  for hyperbolic programs has been developed \cite{Gul,Ren,Ren2} which extends many features of semidefinite programming. 

An important open question regarding hyperbolic programming concerns the generality of hyperbolicity cones. 
The most fundamental example of a hyperbolic polynomial is the determinant $h(\xx) = \det(X)$, where $X=(x_{ij})_{i,j=1}^m$ is the symmetric $m \times m$ matrix of $m(m+1)/2$ variables and $\ee = I$ is the identity matrix. Hence the cone of positive semidefinite $m \times m$ matrices is a hyperbolicity cone, and it follows that so  are \emph{spectrahedral cones}, i.e., cones of the form 
\begin{equation}\label{slice}
\left\{ \xx \in \RR^n : \sum_{i=1}^n  x_i A_i \mbox{ is positive semidefinite}\right\}, 
\end{equation}
where $A_i$, $1\leq i \leq n$, are symmetric $m \times m$ matrices such that there exists a vector $(y_1,\ldots, y_n) \in \RR^n$ such that $\sum_{i=1}^n  y_i A_i$ is positive definite. It has been speculated whether the converse is true \cite{HV,Ren}. 
\begin{conjecture}[Generalized Lax conjecture]\label{lax}
All hyperbolicity cones are spectrahedral i.e., of the form \eqref{slice}. 
\end{conjecture}
The evidence in favor of Conjecture \ref{lax} are not overwhelming: 
\begin{enumerate}
\item It is true for hyperbolic polynomials in three variables \cite{HV,LPR},
\item It is true for quadratic polynomials \cite{NT}.
\end{enumerate}
Stronger conjectures that imply Conjecture \ref{lax} were recently disproved in \cite{Br1}, see also \cite{NT}. 

In this note we are concerned with the hyperbolicity cones of the \emph{elementary symmetric polynomials}: 
$$
e_k(x_1,\ldots, x_n)= \! \! \! \!  \sum_{1\leq i_1< \cdots <i_k \leq n} \! \! \! \! \!  x_{i_1}\cdots x_{i_k}. 
$$
They are hyperbolic with respect to $\one=(1,\ldots,1)^T$, and their hyperbolicity cones contain the positive orthant. Zinchenko \cite{Zi} studied the hyperbolicity cones of elementary symmetric polynomials and proved that they are \emph{spectrahedral shadows}, i.e., projections of spectrahedral cones. Sanyal \cite{Sanyal} proved that the hyperbolicity cone of $e_{n-1}(x_1,\ldots, x_n)$ is spectrahedral and conjectured that all hyperbolicity cones of elementary symmetric polynomials are spectrahedral, although it is subsumed by Conjecture \ref{lax}. We will prove this conjecture. 
\begin{theorem}\label{eldet}
Hyperbolicity cones of elementary symmetric polynomials are spectrahedral. 
\end{theorem}
Note that the hyperbolicity cone, with respect to $\ee=(e_1,\ldots, e_n)$, of $h(\xx)$ is spectrahedral if there is a pencil $\sum_{i=1}^n  x_i A_i$ of symmetric matrices (such that $\sum_{i=1}^n  e_i A_i$ is positive definite) and a homogeneous polynomial $q(\xx)$ such that 
\begin{equation}\label{fac}
q(\xx)h(\xx)= \det\left(\sum_{i=1}^n  x_i A_i \right),
\end{equation}
and $\Lambda_+(q,\ee)  \supseteq \Lambda_+(h,\ee)$. 
Our idea for proving Theorem \ref{eldet} was to use the matrix--tree theorem (Theorem \ref{matrix-tree} below) to construct a graph for which the spanning tree polynomial is a multiple of the elementary symmetric polynomial in question. In the process we became aware of that the same idea was already present in \cite[Section 9.1]{COSW} where it was  observed that the elementary symmetric polynomials are factors of determinantal polynomials. However, to prove Theorem \ref{eldet} we need to know the other factors, and that $\Lambda_+(q,\ee)  \supseteq \Lambda_+(h,\ee)$ holds in \eqref{fac}. 

Recall that a  cone is \emph{polyhedral} if  it is the intersection of a finite number of half-spaces, i.e., if it is the hyperbolicity cone of a polynomial of the form $h(\xx)= \ell_1(\xx)\cdots \ell_d(\xx)$ where $\ell_j(\xx)$ is a linear form for $1 \leq j \leq d$. If $h$ is hyperbolic with respect to $\ee=(e_1,\ldots, e_n)^T$, then 
$$
D_\ee h(\xx) = \sum_{i=1}^ne_i \frac {\partial}{\partial x_i} h(\xx)
$$
is hyperbolic with respect to $\ee$ and $\Lambda_+(D_\ee h,\ee)  \supseteq \Lambda_+(h,\ee)$, see e.g. \cite{Ga,Ren}. Of course polyhedral cones are spectrahedral and it is natural to ask if the derivative cones 
$\Lambda_+(D_\ee^k h,\ee)$ are also spectrahedral for $1 \leq k \leq d-1$. For $k=1$  this was answered to the affirmative by Sanyal \cite{Sanyal}. Using Theorem \ref{eldet} we settle the remaining cases. 
\begin{corollary}
The derivative cones of polyhedral cones are spectrahedral, i.e., if 
$h(\xx)= \ell_1(\xx)\cdots \ell_d(\xx)$ and $h(\ee) \neq 0$, then 
 $\Lambda_+(D_\ee^k h,\ee)$ is spectrahedral for each $1 \leq k \leq d-1$. 
\end{corollary}

\begin{proof}
Since 
$$
D_\ee^kh(\xx) = k!h(\ee)e_{d-k}(\ell_1(\xx), \ldots, \ell_d(\xx)),
$$
see e.g. \cite[Proposition 18]{Ren}, the corollary follows immediately from Theorem~\ref{eldet}. \qed
\end{proof}

\section{Proof of Theorem \ref{eldet}}
Let $G=(V,E)$ be a finite graph where $V$ is the set of vertices, $E$ is the set of edges, and each edge connects two distinct vertices. We allow for more than one edge connecting two distinct vertices. The graphs considered here are connected, i.e., between each pair of distinct vertices there is a path connecting  them. Assign variables $\xx=\{x_e\}_{e \in E}$ to the edges. Recall that a \emph{spanning tree}  is a maximal  (with respect to inclusion) subset $T$ of $E$ that contains no cycle, i.e., a minimal (with respect to inclusion)  set $T \subseteq E$ such that the graph $(V,T)$ is connected. The \emph{spanning tree polynomial} is defined as 
$$
T_G(\xx)= \sum_T \prod_{e \in T}x_e,
$$
where the sum is over all spanning trees in $G$, see Fig. \ref{Figu4}. 
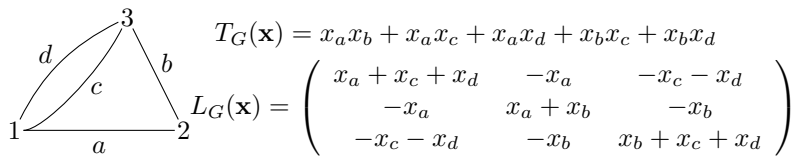
\begin{figure}[htp]
$$
\begin{tikzpicture}[scale=0.75, inner sep=0.7pt]
  \node (0) at (0,0) {$1$};
  \node (1) at (3,0) {$2$};
    \node (2) at (2,2) {$3$};
    \node (3) at (1.5,-.3) {$a$};
    \node (4) at (1.45,.75) {$c$};
     \node (5) at (2.7,1.2) {$b$};
      \node (6) at (0.55,1.35) {$d$};
       \node (7) at (8,1.7) {$T_G(\xx)=x_ax_b+x_ax_c+x_ax_d+x_bx_c+x_bx_d$};
        \node (8) at (8.5,0.4) {$
L_G(\xx)=  \left( \begin{array}{ccc}
x_a+x_c+x_d & -x_a & -x_c-x_d \\
-x_a & x_a+x_b & -x_b\\
-x_c-x_d & -x_b & x_b+x_c+x_d \end{array} \right)
$};
\draw (0) -- (1) -- (2);
\draw (0) .. controls (0.5,1) and (1,1.5) .. (2); 
\draw (0) .. controls (.5,0) and (1.5,1) .. (2); 
\end{tikzpicture}
$$ 
\caption{A graph $G$, its spanning tree polynomial and its weighted Laplacian. \label{Figu4}}
\end{figure}  
Suppose $V=[n]:=\{1,\ldots, n\}$ and let $\{\delta_i\}_{i=1}^n$ be the standard bases of $\RR^n$. The \emph{weighted Laplacian} of 
$G$ is defined as 
$$
L_G(\xx)= \sum_{e\in E}x_e(\delta_{e_1}-\delta_{e_2})(\delta_{e_1}-\delta_{e_2})^T,
$$
where $e_1$ and $e_2$ are the vertices incident to $e\in E$. In other words if $L_G(\xx)=(v_{ij}(\xx))_{i,j=1}^n$, then 
$
v_{ii}(\xx)= \sum_{e}x_e 
$,
where the sum is over all edges containing $i$, and if $i\neq j$, then 
$
 v_{ij}(\xx)= -\sum_{e} x_e
$,
where the sum is over all edges connecting $i$ and $j$, see Fig. \ref{Figu4} for an example. 
We refer to \cite[Theorem VI.29]{Tutte} for a proof of the next classical theorem that goes back to Kirchhoff and Maxwell.
\begin{theorem}[Matrix--tree theorem]\label{matrix-tree}
For $i \in V$, let $L_G(\xx)_{ii}$ be the matrix obtained by deleting the column and row indexed by $i$ in $L_G(\xx)$.  Then 
$$
T_G(\xx) = \det(L_G(\xx)_{ii}).
$$
\end{theorem}

Note that the matrix--tree theorem implies that the hyperbolicity cone of any connected graph $G$ is spectrahedral. Indeed $L_G(\xx)_{ii}$ is a pencil of positive semidefinite  matrices, and since $\det L_G(\one)_{ii}$ is equal to the number of spanning trees of $G$ we see that $L_G(\one)_{ii}$ is positive definite. 

\begin{remark}
Let $G=K_{n+1}$, the complete graph on $n+1$ vertices. Then $${L_G(\xx)}_{(n+1)(n+1)}=(v_{ij})_{i,j=1}^n,$$ where 
$v_{ij}= -x_{ij}$ if $i \neq j$ and $v_{ii}= x_{i(n+1)}-x_{ii}+\sum_{j=1}^{n}x_{ij}$.  Hence  the hyperbolicity cone of $T_G(\xx)$ is linearly isomorphic to the cone of positive semidefinite $n \times n$ matrices. Thus the generalized Lax conjecture is equivalent to that each hyperbolicity cone is a  slice of a hyperbolicity cone of some spanning tree polynomial. This is the reason for why we believed that at least for elementary symmetric polynomials one would be able to use the matrix--tree theorem to deduce that the hyperbolicity cones are spectrahedral. 
\end{remark}

The plan is to construct polynomials, $H_{k,k}(\xx)$, which are obtained from the  spanning tree polynomials of graphs $G_{k,k}$ by linear changes of variables. We will prove that $H_{k,k}(\xx)$ contains the elementary symmetric polynomial $e_{k+1}(\xx)$ as a factor, and  that the hyperbolicity cones of the other factors contain the hyperbolicity cone of $e_{k+1}(\xx)$. To do this we explicitly compute $H_{k,k}(\xx)$ (in Lemma \ref{Hkk}) and observe that all factors except $e_{k+1}(\xx)$ are directional derivatives of $e_k(\xx)$. We begin with a few technical definitions that are essential to the recursive construction used to compute $H_{k,k}(\xx)$.

Let $\{x_j\}_{j=1}^\infty$ be independent variables, and for a finite non-empty set $S \subset \ZZ_+ := \{1,2,\ldots\}$ let 
$$
e_k(S)= \mathop{\sum_{T \subseteq S}}_{|T| = k} \prod_{j \in T} x_j,
$$
be the $k$th elementary symmetric polynomial in $\{x_j\}_{j \in S}$. For $k \geq 1$ let 
$$
q_k(S)= \frac {e_k(S)}{e_{k-1}(S)}.
$$
From the recursions 
$$
ke_k(S)= \sum_{j \in S} x_j e_{k-1}(S\setminus \{j\}) \quad \mbox{ and } \quad e_k(S) = e_k(S\setminus\{j\})+ x_je_{k-1}(S\setminus\{j\}),
$$
we obtain 
\begin{align}\label{engine}
kq_k(S) &= \sum_{j \in S} \frac{ x_j e_{k-1}(S\setminus \{j\})} {e_{k-1}(S\setminus\{j\})+ x_je_{k-2}(S\setminus\{j\})}\nonumber\\
 &= \sum_{j \in S} \frac {x_j q_{k-1}(S\setminus\{ j\})}{x_j+q_{k-1}(S\setminus\{ j\})},
\end{align}
for all $k \geq 2$. 
It is no accident that  \eqref{engine} is reminiscent of the operation (C) on spanning tree polynomials below.
\begin{itemize}
\item[(A)]
If we replace an edge $e \in E$ between vertices $i$ and $j$ with $k$ parallel edges $e_1,\ldots, e_k$ between $i$ and $j$, then  the resulting polynomial is obtained by setting $x_e=x_{e_1}+\cdots + x_{e_k}$ in $T_G$. 
\item[(B)] If an edge $e$ between $i$ and $j$ is replaced by a path $i, e', k, e'', j$, then the resulting polynomial is 
obtained by multiplying by $x_{e'}+x_{e''}$ and setting 
$$
x_e=\frac{x_{e'}x_{e''}}{ x_{e'}+x_{e''}}
$$
in $T_G(\xx)$. 
\item[(C)] By (A) and (B), if we replace an edge $e$ between $i$ and $j$ by a series parallel graph as in Fig.~\ref{Figu} with edge-variables $x_1, y_1, \ldots, x_m, y_m$, then the resulting polynomial is obtained by multiplying by $\prod_{j=1}^m(x_j+y_j)$ and setting 
$$
x_e= \sum_{j=1}^m \frac {x_j y_j}{x_j+y_j},
$$
in $T_G(\xx)$. Indeed, replace $e$ by $m$ parallel edges (using (A)) and then split each new edge in two using (B). 
\end{itemize}
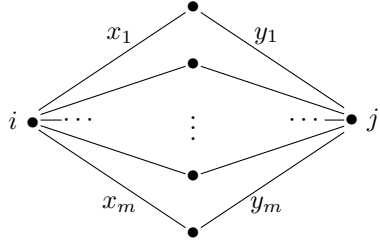
\begin{figure}[htp]
$$
\begin{tikzpicture}[scale=0.75, inner sep=0.2pt]
  \node (0) at (0,2) {$i \ \bullet$};
  \node (1) at (3,4) {$\bullet$};
    \node (9) at (1.7,3.5) {$x_1$};
    \node (19) at (1.7,0.5) {$x_m$};
    \node (10) at (4.3,3.5) {$y_1$};
    \node (11) at (4.3,0.5) {$y_m$};
  \node (2) at (3,3) {$\bullet$};
  \node (3) at (3,2) {$\vdots$};
   \node (4) at (3,1) {$\bullet$};
    \node (5) at (3,0) {$\bullet$};
     \node (6) at (6,2) {$\bullet \ j$};
     \node (7) at (1,2) {$\cdots$};  
       \node (8) at (5,2) {$\cdots$};
  \draw (0) -- (1) -- (6);
   \draw (0) -- (2) -- (6);
   \draw (0) -- (7);
   \draw (8) -- (6);
     \draw (0) -- (4) -- (6);
      \draw (0) -- (5) -- (6);
\end{tikzpicture}
$$ 
\caption{The graph in (C). \label{Figu}}
\end{figure}
If we set $y_j = q_1([m]\setminus \{j\})=e_1([m]\setminus \{j\})$ in Fig. \ref{Figu}, then by (C) and \eqref{engine} we obtain the polynomial 
$$
2q_2([m]) \prod_{j=1}^m (x_j + e_1([m]\setminus \{j\}))= 2 e_2(\xx)e_1(\xx)^{m-1}, 
$$ 
which by the matrix--tree theorem proves that $e_2(\xx)e_1(\xx)^{m-1}$ is a determinantal polynomial, and that the hyperbolicity cone of $e_2(\xx)$  is spectrahedral. We now  extend this construction to higher degrees.

We recursively construct a family of graphs, $\{G_{n,k}\}$, for integers $n \geq k \geq0$, using (C). 
Let $G_{n,0} = s   \bullet \! \! - \! \! \bullet  z$ for all $n \in \ZZ_+$, i.e., the graph with two vertices $s$ and $z$ connected by an edge. For $k>0$, the graph $G_{n,k}$ is constructed from $G_{n,k-1}$ by replacing each edge in $G_{n,k-1}$ that contains $z$ by a graph as in Fig. \ref{Figu}, with $m= n-k$. 
We will need a more explicit description of $G_{n,k}$. For $n\geq k \geq 1$, $G_{n,k}$ is the graph with vertices consisting of two designated vertices $s$ and $z$, and all words $w= w_1w_2\cdots w_\ell$ such that $1 \leq \ell \leq k$, $w_i \in [n]$ for all $i$, and $w_i \neq w_j$ for all $1 \leq i < j \leq \ell$.  The edges in $G_{n,k}$ are between the vertices:
\begin{enumerate}
\item $s$ and $i$ for all $1\leq i \leq n$;
\item  $w_1\cdots w_{i-1}$ and $w_1 \cdots w_{i-1} w_{i}$ for all $2 \leq i \leq k$;
\item  $w_1\cdots w_k$ and $z$,
\end{enumerate} 
see Fig. \ref{Figu2}. 
\begin{figure}[htp]
$$
\begin{tikzpicture}[scale=0.5, inner sep=1pt]
  \node (0) at (0,4) {$s$};
  \node (1) at (3,1) {$1$};
    \node (2) at (3,4) {$2$};
    \node (3) at (3,7) {$3$};
     \node (123) at (6,4) {$z$};
  \draw (0) -- (1) --  (123) -- (2) -- (0) -- (3) -- (123) ;
\end{tikzpicture} \ \ \ \ \ \ \ \ \ 
\begin{tikzpicture}[scale=0.4, inner sep=1pt]
  \node (0) at (0,4) {$s$};
  \node (1) at (3,1) {$1$};
    \node (2) at (3,4) {$2$};
    \node (3) at (3,7) {$3$};
    \node (12) at (8,0) {$12$};
    \node (13) at (8,2) {$13$};
  \node (21) at (8,3) {$21$};
  \node (23) at (8,5) {$23$};
   \node (31) at (8,6) {$31$};
    \node (32) at (8,8) {$32$};
     \node (123) at (14,4) {$z$};
  \draw (0) -- (1) -- (12) -- (123) -- (32) -- (3) -- (0) -- (2) -- (21) -- (123);
  \draw (1) -- (13) -- (123);
  \draw (2) -- (23) -- (123);
  \draw (3) -- (31) -- (123);
\end{tikzpicture}
$$ 
\caption{The graphs $G_{3,1}$ and  $G_{3,2}$. \label{Figu2}}
\end{figure}
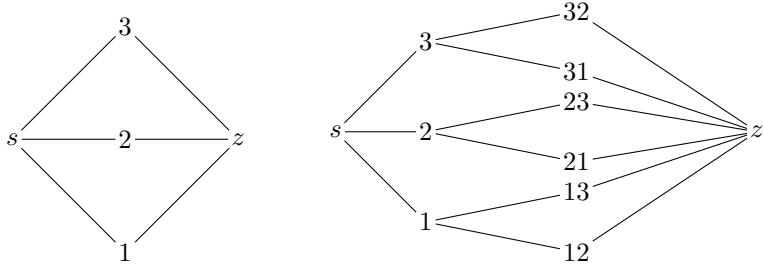

Let $n$ be fixed. For all $r \geq k$ define a rational function $H_{k,r}(\xx)$, by setting the edge-variables in $T_{G_{n,k}}$  as: 
\begin{itemize}
\item[(a)] $r! x_i$ if as in 1;
\item[(b)] $(r-i+1)!x_{w_i}$ if as in 2; 
\item[(c)] $(r-k+1)! q_{r-k+1}(\{w_1,\ldots, w_k\}')$, where $S' =[n]\setminus S$, if as in 3.
\end{itemize}
 Note that if $k=0$ in (c), then the edge variable is $$(r+1)! q_{r+1}([n])=(r+1)!\frac{e_{r+1}(\xx)}{e_r(\xx)}.$$ We are only interested in $H_{k,k}(\xx)$, but to get the recursion running smoothly we need the extra parameter $r$.
\begin{example}
Let us compute $H_{2,2}(\xx)$ for $n=4$. For each $j=1,2,3,4$, replace the pieces between $j$ and $z$ in $G_{4,2}$ by a new edge using (C) and \eqref{engine}, see Fig.~\ref{Figu3}. The new edge variable for the edge $\{4,z\}$ is
$$
\frac{x_3(x_1+x_2)}{x_3+x_1+x_2}+\frac{x_2(x_1+x_3)}{x_2+x_1+x_3}+ \frac{x_1(x_2+x_3)}{x_1+x_2+x_3} = 2q_2([4]\setminus \{4\}), 
$$   
and thus $2q_2([4]\setminus \{j\})$ for the edge $\{j,z\}$. We are left with a graph as in Fig.~\ref{Figu} with edge weights $2x_j$ and $2q_2([4]\setminus \{j\})$, respectively. If we replace this graph by a new edge using (C) and \eqref{engine} we arrive at the graph $s   \bullet \! \! - \! \! \bullet  z$ with edge weight $6q_3([4])$. By multiplying with the factors for each time  (C) was used we arrive at
\begin{align*}
H_{2,2}(\xx)&= 6q_3([4]) \prod_{j=1}^4 \left( 2x_j + 2 \frac {e_2([4]\setminus \{j\})}{e_1([4]\setminus \{j\})} \right) \prod_{j=1}^4 e_1([4]\setminus \{j\})^3\\
&= 96e_3([4])e_2([4])^3 \prod_{j=1}^4 e_1([4]\setminus \{j\})^2.
\end{align*}

\begin{figure}[htp]
$$
\begin{tikzpicture}[scale=0.5, inner sep=1pt]
  \node (0) at (0,4) {$4$};
  \node (1) at (3,1) {$41$};
    \node (2) at (3,4) {$42$};
    \node (3) at (3,7) {$43$};
     \node (123) at (6,4) {$z$};
  \node (3z)   at (5.4,6) {$x_1+x_2$};
   \node (1z)   at (5.4,2) {$x_2+x_3$};
   \node (41)   at (1.2,2) {$x_1$}; 
   \node (43)   at (1.2,6) {$x_3$}; 
   \node (42)   at (1.4,4.3) {$x_2$}; 
    \node (s4)   at (-1.3,4.1) {$2x_4$}; 
    \node (2z)   at (4.4,4.4) {$x_1+x_3$}; 
  \draw (-2.1,3.5) -- (0) -- (1) --  (123) -- (2) -- (0) -- (3) -- (123) ;
\end{tikzpicture}
$$ 
\caption{A piece of $G_{4,2}$ with $k=r=2$. \label{Figu3}}
\end{figure}
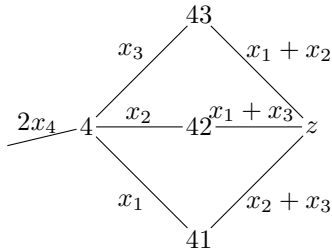
\end{example}
 Let 
$$
\gamma_{kr} = \prod_{S \in \binom {[n]}{n-k}}e_{r-k}(S)^{k!} = \prod_{S \in \binom {[n]}{k}}e_{r-k}(S')^{k!} , 
$$
where $0 \leq k \leq r \leq n$ and $\binom {[n]}{j} = \{ U \subseteq [n] : |U|=j\}$. 
\begin{lemma}\label{step}
Let $1\leq k \leq r \leq n-1$ be integers. Then there are positive constants $C_{i,r}$, $0 \leq i \leq r$, such that 
$$
H_{0,r}(\xx) = C_{0,r} \frac{e_{r+1}(\xx)}{e_r(\xx)}, 
$$
and  
$$
H_{k,r}(\xx) = C_{k,r}H_{k-1,r}(\xx) \frac { (\gamma_{(k-1)r})^{n-k+1} }{\gamma_{kr}}, 
$$
for $k > 0$. 
\end{lemma}
\begin{proof}
The first statement follows immediately from the definitions. For $k>0$ consider the graph  $G_{n,k}$ with edge variables as described in (a), (b) and (c) above. Now consider a right-most piece of this graph  as depicted in Fig. \ref{Figu5}. 
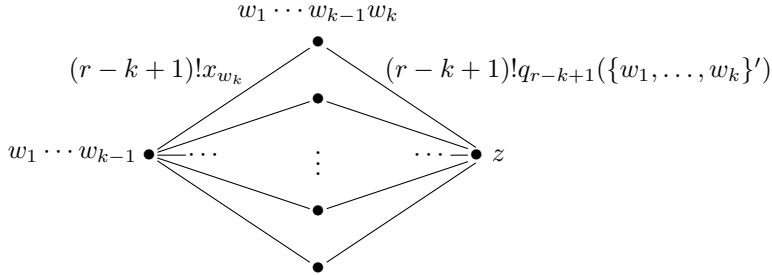
\begin{figure}[htp]
$$
\begin{tikzpicture}[scale=0.75, inner sep=.5pt]
\node (l) at (-1.2,2) {  $w_1\cdots w_{k-1}$ \ \ };
  \node (0) at (0,2) {  $\bullet$};
  \node (1) at (3,4) {$\bullet$};
  \node (k) at (3,4.5) {$w_1\cdots w_{k-1}w_k$};
    \node (9) at (.3,3.5) {$(r-k+1)!x_{w_k}$ \ \ };
    \node (19) at (1.7,0.5) {};
    \node (10) at (7.3,3.5) { \ \ \ \ $(r-k+1)!q_{r-k+1}(\{w_1,\ldots,w_k\}')$};
    \node (11) at (4.3,0.5) {};
  \node (2) at (3,3) {$\bullet$};
  \node (3) at (3,2) {$\vdots$};
   \node (4) at (3,1) {$\bullet$};
    \node (5) at (3,0) {$\bullet$};
     \node (6) at (6,2) {$\bullet \ z$};
     \node (7) at (1,2) {$\cdots$};  
       \node (8) at (5,2) {$\cdots$};
  \draw (0) -- (1) -- (6);
   \draw (0) -- (2) -- (6);
   \draw (0) -- (7);
   \draw (8) -- (6);
     \draw (0) -- (4) -- (6);
      \draw (0) -- (5) -- (6);
\end{tikzpicture}
$$ 
\caption{A right-most piece of  $G_{n,k}$. \label{Figu5}}
\end{figure}
If we replace this piece by an edge using (C) and \eqref{engine}, then the new edge variable is $(r-(k-1)+1)!q_{r-(k-1)+1}(\{w_1,\ldots, w_{k-1}\}')$ which is the correct edge variable as in (c) for $G_{n,k-1}$. If we do this for each rightmost subgraph such as in Fig. \ref{Figu5} we thus get 
$$
H_{k,r}(\xx) = Q_{k,r}(\xx)H_{k-1,r}(\xx), 
$$
where $Q_{k,r}(\xx)$ is the product of all the factors that come from the operations as in (C). Modulo constants, for each word $w_1\cdots w_k$ we get a factor 
$$
x_{w_k}+ q_{r-k+1}(\{w_1,\ldots, w_k\}')= \frac { e_{r-k+1}(\{w_1,\ldots, w_{k-1}\}')} {e_{r-k}(\{w_1,\ldots, w_k\}')}.
$$
For each $S \in \binom {[n]} {k-1}$, the numerator $e_{r-k+1}(S')$ will appear exactly $(k-1)!(n-k+1)$ times in the product, since there are $(k-1)!$ ways to linearly order $S$ to obtain a word $w_1\cdots w_{k-1}$, and then $n-k+1$ choices for $w_k$.  This accounts for the term $(\gamma_{(k-1)r})^{n-k+1}$. Similarly, for each $S \in \binom {[n]} {k}$ the denominator $e_{r-k}(S')$ will appear exactly $k!$ times, since there are $k!$ ways to linearly order $S$ to obtain a word $w_1\cdots w_{k}$. This accounts for the term $\gamma_{kr}$ in the denominator.\qed
\end{proof}

\begin{lemma}\label{Hkk}
Let $1 \leq k \leq n-1$. Then 
\begin{equation}\label{stor}
H_{k,k}(\xx)= C_ke_{k+1}(\xx)\mathop{\prod_{S \subseteq [n]}}_{|S| \leq k-1}  (\partial^Se_{k}(\xx))^{|S|!(n-|S|-1)},
\end{equation}
where $C_k$ is a positive constant and $\partial^S = \prod_{j \in S}\partial / \partial x_j$.
\end{lemma}

\begin{proof}
By iterating Lemma \ref{step}, modulo constants, 
\begin{align*}
H_{r,r}(\xx) &=  \frac {(\gamma_{(r-1)r})^{n-r+1}}{\gamma_{rr}} \frac {(\gamma_{(r-2)r})^{n-r+2}}{\gamma_{(r-1)r}} \cdots \frac {(\gamma_{1r})^{n-1}}{\gamma_{2r}}\frac {(\gamma_{0r})^{n}}{\gamma_{1r}}H_{0,r}(\xx) \\
&=  e_{r+1}(\xx) \prod_{j=0}^{r-1} \gamma_{jr}^{n-j-1},  
\end{align*}
where we have used $\gamma_{0r}= e_r(\xx)$, $\gamma_{rr}=1$ and $H_{0,r}(\xx) = C_{0,r} {e_{r+1}(\xx)}/{e_r(\xx)}$. 
The theorem follows by noting that $\partial^S e_k(\xx)= e_{k-|S|}(S')$. \qed
\end{proof}

To finish the proof Theorem \ref{eldet} we need further properties of hyperbolic polynomials. The next lemma is fundamental (and known) but we could not find a proof in the literature. 
\begin{lemma}\label{prop}
Let $U \subseteq \RR^n$ be an open and  connected  set, and let $\PP_{n,d}(U)$ be the space of all hyperbolic polynomials of degree $d$ in $\RR[x_1, \ldots, x_n]$ with hyperbolicity cone containing $U$, i.e., polynomials that are  hyperbolic with respect to each $\ee \in U$. Then $\PP_{n,d}(U) \cup \{0\}$ is closed (under point-wise convergence). 
\end{lemma}

\begin{proof}
We claim that a homogeneous polynomial $h \in \RR[x_1, \ldots, x_n]$ of degree $d$ belongs to $\PP_{n,d}(U)$ if and only if $h(\zz)\neq 0$ for all $\zz$ in the tube 
$U+i\RR^n := \{\xx+i\yy : \xx \in U \mbox{ and } \yy \in \RR^n\}$. Indeed, if $h$ is hyperbolic with respect to each $\ee \in U$, then 
$h(\ee+i\yy)=(-i)^dh(-\yy+i\ee) \neq 0$ for all $\yy \in \RR^n$ by the definition of hyperbolicity. Conversely,  if $h$ fails to be hyperbolic for some $\ee \in U$, then $h(\xx+(a+ib)\ee)=0$ for some $\xx \in \RR^n$ and $a, b \in \RR$ with $b \neq 0$. Thus, by homogeneity, $h(\ee+i(-b^{-1}\xx-ab^{-1}\ee)=0$, so that 
$h$ fails to be non-vanishing on $U+i\RR^n$. 

If $\{h_k\}_{k=1}^\infty$ is a sequence of polynomials in $\PP_{n,d}(U)$ which converges point-wise to $h$, then the convergence is also uniform on compact subsets of $\CC^n$ (by the equivalence of norms in finite dimensions). 
It now follows from Hurwitz' theorem on the continuity of zeros (see \cite[p. 96]{COSW} for a multivariate version) that either $h \in \PP_{n,d}(U)$, or 
$h \equiv 0$. 
\end{proof}

\begin{lemma}\label{partial}
Suppose $h$ is hyperbolic with respect to $\ee$, and $\vv \in \Lambda_+(h,\ee)$ is such that $D_\vv h \not \equiv 0$. Then $D_\vv h$ is hyperbolic with respect to $\ee$, and $\Lambda_+(h,\ee)\subseteq \Lambda_+(D_\vv h,\ee)$. 
\end{lemma}
\begin{proof}
Let $\vv \in \Lambda_{+}(h,\ee)$ and let $U$ be the interior of $\Lambda_{+}(h,\ee)$. 
If $\vv \in U$, then the conclusion is known to follow, see e.g. \cite[Theorem 4]{Ga}. Otherwise, if $\vv$ is on the boundary of  $\Lambda_{+}(h,\ee)$ take a sequence $\{\vv_k\}_{k=0}^\infty \subset U$  such that $\lim_{k \to \infty} \vv_k = \vv$. Then $D_{\vv_k} h \in \PP_{n,d-1}(U)$ for all $k$, by the above. 
The lemma now follows from Lemma~\ref{prop}. 
\qed
 \end{proof}
\emph{Proof of Theorem \ref{eldet}} 
By the matrix--tree theorem we may write the spanning tree polynomial of $G_{n,k}$ as $\det\!\left(\sum_{e \in E}x_eA_e\right)$, where $\{A_e\}_{e \in E}$ are positive semidefinite, and $A=\sum_{e \in E} A_e$ is positive definite (since $\det(A)$ is positive and equal to the number of spanning trees of the connected graph $G_{n,k}$). 
Note that when $k=r$, the edge-variables in (c) in the construction of $H_{k,k}(\xx)$ are given by 
$$
 \sum_{j  \in [n]\setminus \{w_1, \ldots, w_k\}} \!  \!  \!  \!  \!  x_j. 
$$ 
Hence we may write $H_{k,k}(\xx)=\det\!\left(\sum_{j=1}^nx_j B_j\right)$, where $B_1,\ldots, B_n$ are positive semidefinite. Note also that  
$
\det\!\left(\sum_{j=1}^n B_j\right)= H_{k,k}(\one) \neq 0, 
$
by Lemma~\ref{Hkk}. Hence $\sum_{j=1}^n B_j$ is positive definite, and it 
remains to prove $\Lambda_{+}(e_{k+1}) \subseteq \Lambda_{+}(\partial^Se_{k})$. Now 
$e_k(\xx) =  (n-k)^{-1}D_\one e_{k+1}(\xx)$, so that $\Lambda_+(e_{k+1}) \subseteq \Lambda_+(e_{k})$ by Lemma \ref{partial}. Since the coordinate directions are in $\Lambda_+(e_k)$, we have $\Lambda_+(e_k) \subseteq \Lambda_{+}(\partial^Se_{k})$ by Lemma~\ref{partial}. \qed


\begin{thebibliography}{99}
\bibitem{Br1} P.~Br\"and\'en, {\em Obstructions to determinantal representability,} Adv. Math., {\bf 226} (2011), 1202--1212, \url{http://arxiv.org/pdf/1004.1382.pdf}. 

\bibitem{COSW} 
Y.~Choe, J.~Oxley, A.~Sokal, D.~G.~Wagner, { Homogeneous multivariate 
polynomials with the half-plane property}. 
Adv. Appl. Math. {\bf 32} (2004), 88--187, \url{http://arxiv.org/pdf/math/0202034.pdf}.

\bibitem{Ga} 
L.~G\aa rding, {An inequality for hyperbolic polynomials}, 
J. Math. Mech. {\bf 8} (1959), 957--965. 

\bibitem{Gul} O. G\"uler,  Hyperbolic polynomials and interior point methods for convex programming, Math. Oper. Res., {\bf 22}
(1997), 350--377.

\bibitem{HV} J.~Helton, V.~Vinnikov, Linear matrix inequality representation of sets, Comm. Pure Appl. 
Math. {\bf 60} (2007), 654--674, \url{http://arxiv.org/pdf/math/0306180.pdf}. 


\bibitem{LPR} A.~Lewis, P.~Parrilo, M.~Ramana, The Lax conjecture is true, Proc. Amer. Math. Soc. {\bf 133} (2005), 2495--2499, \url{http://arxiv.org/pdf/math/0304104.pdf}. 

\bibitem{NT} T.~Netzer, A.~Thom, Polynomials with and without determinantal representations, Linear Algebra Appl., {\bf 437} (2012), 1579--1595, \url{http://arxiv.org/pdf/1008.1931v2.pdf}.


\bibitem{Ren} J.~Renegar,  Hyperbolic programs, and their derivative relaxations,  Found. Comput. Math., 
{\bf 6} (2006), 59--79.

\bibitem{Ren2} J.~Renegar, Central swaths: a generalization of the central path, Found. Comput. Math., {\bf 13} (2013), 405--454, \url{http://arxiv.org/pdf/1005.5495.pdf}.

\bibitem{Sanyal} R.~Sanyal, On the derivative cones of polyhedral cones, Adv. Geom. {\bf 13} (2013), 315--321, \url{http://http://arxiv.org/pdf/1105.2924v2.pdf}.


\bibitem{Tutte} W.~T.~Tutte,  Graph Theory,  Encyclopedia of Mathematics and its Applications, {\bf 21}, Addison-Wesley, 1984.

\bibitem{Zi} Y.~Zinchenko, On hyperbolicity cones associated with elementary symmetric polynomials, Optim. Lett., {\bf 2} (2008), 389--402.

\end{thebibliography}
\end{document}